\documentclass[oneside,english]{amsart}
\usepackage[T1]{fontenc}
\usepackage[latin9]{inputenc}
\usepackage{amsthm}
\usepackage{amstext}
\usepackage{amssymb}

\numberwithin{equation}{section}
\numberwithin{figure}{section}
\theoremstyle{plain}
\newtheorem{thm}{Theorem}
\theoremstyle{plain}
\newtheorem*{prop*}{Proposition}
\theoremstyle{definition}
\newtheorem{defn}[thm]{Definition}
\theoremstyle{remark}
\newtheorem*{rem*}{Remark}
\theoremstyle{remark}
\newtheorem{rem}{Remark}
\theoremstyle{plain}
\newtheorem{lem}[thm]{Lemma}
\theoremstyle{remark}
\newtheorem*{claim*}{Claim}
\theoremstyle{plain}
\newtheorem{prop}[thm]{Proposition}
\theoremstyle{plain}
\newtheorem{cor}[thm]{Corollary}
\theoremstyle{definition}

\newcommand{\mathcircumflex}[0]{\mbox{\^{}}}

\title{Dimension in TTT Structures}

\author{Daniel Lowengrub}

\address{Department of Mathematics\\ The Hebrew University of Jerusalem\\
  Jerusalem, 91904, Israel}

\email{lowdanie@gmail.com}

\begin{document}

\maketitle

\begin{abstract}
  In this paper we consider two types of dimension that can be defined
  for products of one-dimensional topologically totally transcendental (t.t.t) structures. The first
  is topological and considers the interior of projections of the set onto 
  lower dimensional products.  The second one is based on algebraic 
  dependence. We show that these definitions are equivalent for $\omega$-saturated
  one-dimensional t.t.t structures. We also prove that sets which are dense in products of
  these structures are comeager.
\end{abstract}

\section{Introduction}
There are a several of different ways to think about dimension in finite products of o-minimal structures.
One of the first ways to do this was described in \cite{KPS} using the cell decomposition theorem.
The idea is to first define dimension on relatively  simple sets called cells,
and then to generalize this to arbitrary sets by relying on the fact that any set can
be decomposed into a finite number of cells.

An alternative approach was introduced by Pillay in \cite{P2}. This definition is of
a more algebraic flavor and it is based on the notion of algebraic dependence. Assuming
that this dependence is well behaved, we can obtain a concept of dimension
in a way that is quite similar to what is done in the case of linear independence in 
linear algebra.

A third possibility is to take a topological route and define the dimension of a set $X$
as the largest integer $k\in\mathbb{N}$ such that some projection onto $M^k$ has an 
interior.

It is natural to ask whether the various definition coincide. We can also ask whether
it is possible to extend these types of dimension beyond the o-minimal setting. And assuming we can,
how how does this affect their relationship with one another?

In \cite[1.4]{P2} Pillay showed that for o-minimal structures, the second and third definitions
are equivalent. As he mentions in \cite[1.5]{P3}, they are also equivalent to the first
definition.

A natural generalization of o-minimal structures is that of a \emph{first order topological structure}
which was introduced by Pillay \cite{P1}. In particular, a subset of these structures called
\emph{one dimensional topologically totally transcendental} (1-t.t.t) structures share several important
characteristics with the o-minimal ones. For example, they have the exchange property which allows
us to define dimension using algebraic dependence.

Mathews proved in \cite[8.8]{M} that the equivalence mentioned above holds in
the generalized setting of first order topological structures which have both the exchange
property and what he defined as the cell decomposition property.

In this paper we prove the equivalence of the second two kinds of dimension for $\omega$-saturated
1-t.t.t structures. We note that for $\omega$-saturated connected first order 
topological structures, Mathew's cell decomposition implies 1-t.t.t.

In addition, we'll use some of the machinery developed during
the proof in order to obtain additional information about dense sets in this type of structure.
Specifically, we'll show that dense sets must be comeager.

\begin{prop}\label{prop:dense-is-comeager}
  Suppose $M$ is an $\omega$-saturated 1-t.t.t structure. Let $X\subset M_{t}^{n}$ be
  a dense definable set. Then $int(X)\subset M_{t}^{n}$ is dense as
  well. 
\end{prop}

\section{Preliminaries}
\label{sec:preliminaries}

We start by presenting the definitions from \cite{P1} necessary for introducing t.t.t structures.

\begin{defn}\label{defn:first-order-top-th}
  Let $M$ be a two sorted $L$ structure with sorts $M_{t}$ and $M_{b}$
  and let $\phi(x,y_{1},\dots,y_{k})$ be an $L$ formula such that
  $\{\phi^{M_{t}}(x,\bar{a})\vert\bar{a}\in M_{b}^{k}\}$ is a basis
  for a topology on $M_{t}$. Then the pair $(M,\phi)$ will be called
  a \emph{first order topological structure}. When we talk about the
  topology of $M_{t}$ we mean the one generated by the basis described
  above.
\end{defn}

We'll also be using the following property:
 
(A) Every definable set $X\subset M_{t}$ is a boolean combination
of definable open subsets.

\begin{defn}\label{defn:dimension}
  Let $M$ be a first order topological structure satisfying (A) such that $M_t$ is Hausdorff and let
  $X\subset M_{t}$ be a closed definable subset of $M_{t}$. The ordinal
  valued $D_{M}(X)$ is defined by:
  \begin{enumerate}
  \item If $X\neq\emptyset$ then $D_{M}(X)\geq0$.
  \item If $\delta$ is a limit ordinal and $D_{M}(X)\geq\alpha$ for all
    $\alpha<\delta$ then $D_{M}(X)\geq\delta$.
  \item If there's a closed definable $Y\subset M_{t}$ such that $Y\subset X$,
    $Y$ has no interior in $X$ and $D_{M}(Y)\geq\alpha$ then $D_{M}(X)\geq\alpha+1$.
  \end{enumerate}

  Furthermore, we'll write $D_{M}(X)=\alpha$ if $D_{M}(X)\geq\alpha$ and 
  $D_{M}(X)\ngeq\alpha+1$.
  We'll write $D_{M}(X)=\infty$ if $D_{M}(X)\geq\alpha$ for all $\alpha$.
\end{defn}

\begin{defn}\label{defn:has-dim}
  Let $M$ be a first order topological structure satisfying (A) such that $M_t$ is Hausdorff. 
  We say that $M$ \emph{has dimension }if $D_{M}(X)\neq\infty$ for
  all closed definable subsets $X\subset M_{t}$.
\end{defn}

\begin{defn}\label{defn:connected-components}
  Let $M$ be a first order topological structure satisfying (A) such that $M_t$ is Hausdorff. 
  Let $X\subset M_{t}$ be a definable subset. Then $d_{M}(X)$ is the maximum
  $d<\omega$ such that there are disjoint definable clopen $X_{1},\dots,X_{d}\subset X$
  with $X=\cup_{i=1}^{d}X_{i}$ , and $\infty$ if no such $d$ exists.
\end{defn}

We're now ready to introduce t.t.t structures.

\begin{defn}\label{defn:ttt}
  We say that $M$ is \emph{topologically totally transcendental (t.t.t)}
  if $M$ is a first order topological structure satisfying (A) with
  dimension such that $M_t$ is Hausdorff and for every definable set $X\subset M_{t}$, $d_{M}(X)<\infty$.
  We say that a theory $T$ is t.t.t is every model of $T$ is t.t.t.
\end{defn}

As mentioned in the introduction, 1-t.t.t will stand for \emph{one-dimensional t.t.t}.

The following lemma was proved by Pillay \cite[6.6]{P1} and will be used extensively.

\begin{lem}\label{lem:Let--be}
  Let $M$ be a 1-t.t.t structure. Then:
  \begin{enumerate}
  \item For any closed and definable $X\subset M_{t}$, $D(X)=0$ iff $X$
    is finite.
  \item The set of isolated points of $M_{t}$ is finite.
  \item For any definable $X\subset M_{t}$ there are pairwise disjoint definably
    connected definable open subsets $X_{1},\dots,X_{m}\subset M_{t}$
    and a finite set $Y\subset M_{t}$ such that $X=(\cup_{i=1}^{m}X_{i})\cup Y$.
  \item For any definable $X\subset M_{t}$, the set of boundary points of
    $X$ is finite.
  \end{enumerate}
\end{lem}

\section{Defining Dimension in t.t.t Structures}
\label{sec:defin-dimens-t.t.t}

\begin{defn}\label{defn:exchange}
  \emph{(exchange) }Let $M$ be a first order structure. We say that
  $M$ has the \emph{exchange property } if for every $a,b\in M$ and
  a set $A\subset M$, if $b\in acl(A\cup a)$ and $b\notin acl(A)$
  then $a\in acl(A\cup b)$.
\end{defn}

The following theorem was proved by Pillay \cite[6.7]{P1}:
\begin{thm}\label{thm:ttt-exchange}
  Let $M$ be a 1-t.t.t structure. Then $M_{t}$ has the exchange property.
\end{thm}

The first type of dimension that we'll look at was introduced by Pillay \cite{P2}.
For this part we'll leave the t.t.t setting and will only need to assume that our structure
has the exchange property.

\begin{defn}\label{def:rank}
  (\emph{rank) } Let $M$ be a structure with the exchange property and $A\subset M$. 
  \begin{enumerate}
  \item For any tuple $\bar{a}\in M^{n}$, $rk(\bar{a}/A)$ is the least cardinality
    of a subtuple $\bar{a}'$ of $\bar{a}$ such that $\bar{a}\in acl(\bar{a}'/A)$.
  \item for any type $p(\bar{x})\in S_{n}(A)$, $rk(p/A)=rk(\bar{a}/A)$ for
    any $\bar{a}\in M^{n}$ realizing $p$.
  \end{enumerate}
\end{defn}

\begin{rem*}\label{rem:rank-rem}
  It's easy to see that the second part of the definition doesn't depend
  on the choice of the element which realizes $p$.
\end{rem*}

The following lemma is immediate but will be used in the next section.

\begin{lem}\label{lem:independent}
  Let $M$ be a structure with the exchange property, $A\subset M$, and
  $\{a_{1},\dots,a_{n}\}\subset M$ an algebraically independent set
  over $A$. In addition, let $b\in M$ have the property that 
  $b\notin acl(\{a_{1},\dots,a_{n}\}/A)$.
  Then $\{a_{1},\dots,a_{n},b\}$ is an algebraically independent set
  over $A$. 
\end{lem}
\begin{proof}
  Suppose for contradiction that $\{a_{1},\dots,a_{n},b\}$ is not algebraically
  independent over $A$. Then there's some $1\leq i\leq n$ such that
  \[
  a_{i}\in acl(\{a_{1}\dots,\hat{a_{i}},\dots,a_{n},b\}/A)\]
  . By the assumption, $a_{i}\notin acl(\{a_{1}\dots,\hat{a_{i}},\dots,a_{n}\}/A)$
  . But since $M$ has the exchange property, this means that 
  $b\in acl(\{a_{1}\dots,a_{i},\dots,a_{n}\}/A)$ which is a contradiction.
\end{proof}

For now we will assume that $M$ is a structure with the exchange
property which is sufficiently saturated so that the dimension 
of a type doesn't depend on the specific model we're using.

The following lemma was proved by Pillay \cite[1.2]{P2}

\begin{lem}\label{lem:properties}
  Let $M$ be a structure with the exchange property,
  $A,B\subset M$, and $\bar{a}\in M^{s},\bar{b}\in M^{t}$. Then:
  \begin{enumerate}
  \item If $A\subset B$ then $rk(\bar{a}/A)\geq rk(\bar{a}/B)$.
  \item $rk(\bar{a}\mathcircumflex\bar{b}/A)=rk(\bar{a}/A\cup\bar{b})+rk(\bar{b}/A)$.
  \item $rk(\bar{a}/A\cup\bar{b})=rk(\bar{a}/A)\iff rk(\bar{b}/A\cup\bar{a})=rk(\bar{b}/A)$.
  \item If $p\in S_{n}(A)$ and $A\subset B$ then there exists a type $q\in S_{n}(B)$
    such that $p\subset q$ and $rk(q/B)=rk(p/A)$.
  \end{enumerate}
\end{lem}

We are now ready to define our first concept of dimension for a structure
with the exchange property.

\begin{defn}\label{defn:rank-of-set}
  Let $M$ be a structure with the exchange property, $X\subset M^{n}$
  a definable subset and $A\subset M$. Then we define:

  \[
  rk(X)=max_{p\in S_{n}(A)}\{rk(p/A)\vert\textrm{p is realized in}\: X\}\]
\end{defn}

\begin{rem}\label{defn:rank-of-set-rem}
  Note that under our assumption that $M$ is sufficiently saturated,
  by part 4 of lemma \ref{lem:properties}, $rk(X)$ doesn't depend
  on the choice of $A$.

  We can make this more explicit by changing the definition and only
  requiring $p$ to be realized $X(N)$ where $N$ is some elementary extension of $M$.
\end{rem}

We'll now give our second definition of dimension. In this definition 
$M$ has to have some definable topology. Therefore,
we'll assume that $M$ is a t.t.t structure.

Furthermore, given a set $X\subset M^{n}$ and indices $1\leq i_{1}<\dots<i_{k}\leq n$,
let $\pi_{i_{1},\dots,i_{k}}(X)$ be the projection of $X$ onto the coordinates
$i_{1},\dots,i_{k}$.

\begin{defn}\label{defn:top-dim}
  Let $M$ be a t.t.t structure and $X\subset M_{t}^{n}$ be a definable
  subset. We define the \emph{topological dimension} of $X$ as:\[
  dim(X)=max_{1\leq k\leq n}\{\exists1\leq i_{1}<\dots<i_{k}\leq n\:\textrm{s.t}\: int(\pi_{i_{1},\dots,i_{k}}(X))\neq\emptyset\}\]

\end{defn}

\section{The Equivalence of the Dimensions}
\label{sec:equiv-dimens}

In this section we'll prove that for any $\omega$-saturated 1-t.t.t structure, 
the two definitions of dimension we gave above agree on all definable sets. 

\begin{lem}\label{lem:existy}
  Let $M$ be a 1-t.t.t structure. Let $\phi(x,y)$ be a formula and $X=\phi^{M_{t}}$. In addition,
  let $U\subset M_{t}$ be a definable open set such that for all $u\in U$,
  $\vert\{y\in M_{t}:(u,y)\in X\}\vert\geq\aleph_{0}$. Then, for every
  $k\in\mathbb{N}$ there exists a $y\in M_{t}$ such that $\vert\{x\in U:(x,y)\in X\}\vert>k$.
\end{lem}

The proof of this lemma is nearly identical to Pillay's proof of the exchange
property in \cite[6.7]{P1} but is modified for our purposes.
For convenience we give the complete proof here.
  
\begin{proof}
  Let's assume for contradiction that there
  exists a $k\in\mathbb{N}$ such that for all $y\in M_{t}$, $\vert\{x\in U:(x,y)\in X\}\vert\leq k$.
  For all $u\in U$ we'll define $X_{u}=\{y\in M:(u,y)\in X\}$.
  $U$ and $X$ are definable and therefore $X_{u}$ is definable as
  well. In addition, we know that for all $u\in U$, $\vert X_{u}\vert\geq\aleph_{0}$.
  So according to lemma \ref{lem:Let--be}, $X_{u}$ contains an open
  set. We now define another set: \[
  X_{0}=\{c\in M_{t}:c\in\overline{X}_{u}\backslash int(X_{u})\:\textrm{for some}\: u\in U\}\]

  First we'll assume that $X_{0}$ is finite and reach a contradiction.
  Since $\vert\{u\in U:(u,y)\in X\}\vert\leq k$ for all $y\in X_0$,
  we have the following: 

  \medskip{}

  ({*}) for only a finite number of $u\in U$ there exists a $c\in X_{0}$
  such that $(u,c)\in X$. 

  \medskip{}

  Let's define $N=(\cup_{u\in U}X_{u})\backslash X_{0}$ and for all
  $u\in U$, $Z_{u}=X_{u}\cap N=X_{u}\backslash X_{0}$. By ({*}), there're
  an infinite number of $u\in U$ such that $Z_{u}\neq\emptyset$. We'll
  now show that for each $u\in U$, $Z_{u}$ is clopen in $N$. First
  of all, $Z_{u}$ is open in $M_{t}$ and therefore it's also open
  in $N$. In addition, if $c$ is a boundary point
  of $Z_{u}$ in $N$ then it's a boundary point of $Z_{u}$ and therefore
  also a boundary point of $X_{u}$. But that means that $c\in X_{0}$
  which is a contradiction to the definition of $N$.

  Now, by our assumption for contradiction, for any distinct $u_{1},\dots,u_{k+1}\in U$, 

  \medskip{}

  ({*}{*}) $\cap_{i=1}^{k+1}Z_{u_{i}}=\emptyset$

  \medskip{}

  We now show that for any $n\in\mathbb{N}$, we can find $n$ clopen
  definable disjoint sets $V_{1},\dots,V_{n}$ where each $V_{i}$ is
  of the form $Z_{u_{1}}\cap\dots,\cap Z_{u_{m}}$ for some $u_{1},\dots,u_{m}\in U$.
  Let $\tilde{U}=\{u\in U:Z_{u}\neq\emptyset\}$. As we mentioned above,
  $\tilde{U}$ is infinite. For $n=1$, We can define $V_{1}=Z_{u}$
  for any $u\in\tilde{U}$. Let's assume that we've already found sets
  $V_{1},\dots,V_{n}$ with the properties mentioned above. We choose
  some $u_{1}\in\tilde{U}$ that isn't used in the definition of any
  of the $V_{i}$. We define $V_{n+1}^{1}=Z_{u_{1}}$. We now construct
  a sequence $V_{n+1}^{i}$ , $1\leq i\leq k$, inductively. We already
  have $V_{n+1}^{1}$. Let's say that we've defined $V_{n+1}^{i}$.
  If there exists some $u_{i+1}\in\tilde{U}$ such that $V_{n+1}^{i}\cap Z_{u_{i+1}}\neq\emptyset$
  then we define $V_{n+1}^{i+1}=V_{n+1}^{i}\cap Z_{u_{i+1}}$. Otherwise,
  we define $V_{n+1}^{i+1}=V_{n+1}^{i}$. We now define $V_{n+1}=V_{n+1}^{k}$.
  According to ({*}{*}), the sequence $V_{1},\dots,V_{n+1}$ now has
  the required properties. But this is a contradiction to the fact that
  $d(N)\in\mathbb{N}$.

  Now we assume that $X_{0}$ is infinite. Let $W_{0}$ be the interior
  of $X_{0}$. For each $u\in U$, let $W_{u}=intX_{u}$. We'll now
  inductively find a sequence $u_{1},u_{2},\dots\in U$ such that for
  all $n\in\mathbb{N}$, 

\[
  W_{0}\cap W_{u_{1}}\cap\dots\cap W_{u_{n}}\neq\emptyset\]
  
  For $n=0$ there's nothing to show. Let's assume that we've found
  some sequence $u_{1,}\dots,u_{n}\in U$ with the desired property. We choose
  an element $c\in W_{0}\cap\dots\cap W_{u_{n}}$. Since $c\in X_{0}$, there
  exists some point $u\in U\backslash\{u_{1},\dots,u_{n}\}$ such that $c$
  is a boundary point of $X_{u}$. But $X_{u}$ has a finite number
  of boundary points and so by the Hausdorffness of $M_{t}$, every
  neighborhood of $c$ contains points in the interior of $X_{u}$.
  Specifically, $(W_{0}\cap\dots\cap W_{u_{n}})\cap W_{u}\neq\emptyset$
  so we can set $u_{n+1}=u$. Now we choose some $y\in W_{0}\cap\dots\cap W_{u_{k+1}}.$
  This means that $(y,u_{i})\in X$ for all $1\leq i\leq k+1$ which
  is a contradiction to our assumption on $X$.
\end{proof}

\begin{prop}\label{prop:2-d-nonempty-int}
  Suppose that $M$ is an $\omega$-saturated 1-t.t.t structure. 
  Let $\phi(x,y)$ be a formula, $X=\phi^{M_{t}}$, 
  and $U\subset M_{t}$ an open definable subset such that 
  \[\vert\{y\in M_{t}:(u,y)\in X\}\vert\geq\aleph_{0}\]
  for all $u\in U$. Then $X\cap(U\times M_{t})$ has a non-empty interior.
\end{prop}

\begin{proof}
  Let $\alpha(x)$ be the formula in $M$ defining $U$.
  \begin{claim*}
    There exists a $y\in M_{t}$ such that $\vert(M_{t}\times\{y\})\cap X\cap(U\times M_{t})\vert\ge\aleph_{0}$.
  \end{claim*}
  \begin{proof}
    Let $n\leq\omega$. According to lemma \ref{lem:existy} there exists
    some $c\in M_{t}$ such that $\vert(M_{t}\times\{y\})\cap X\vert\geq n$.
    So if we define\[
    \psi_{n}(y)=\exists x_{1}\dots\exists x_{n}((\bigwedge_{i\neq j}x_{i}\neq x_{j})\wedge(\bigwedge_{i}(\alpha(x_{i})\wedge\phi(x_{i},y)))\]
    then $M\vDash\psi_{n}[c]$. Since $M$ is $\omega$-saturated, there
    exists some $d\in M_{t}$ such that $M\vDash\psi_{n}[d]$ for all
    $n<\omega$. Therefore, $\vert(M_{t}\times\{d\})\cap X\cap(U\times M_{t})\vert\ge\aleph_{0}$
    which completes the claim.
  \end{proof}
  \begin{claim*}
    There exists an open definable set $V\subset U$ and an infinite number
    of elements $y\in M_{t}$ such that for all $v\in V$, $(v,y)\in X$.
  \end{claim*}
  \begin{proof}
    By the definition of a t.t.t, there exists some formula $\beta(x,y_{1},\dots,y_{k})$
    such that $\{\beta^{M_{t}}(x,\bar{a})\vert\bar{a}\in M_{b}^{k}\}$
    is a basis for the topology on $M_t$. We first show that for every
    $n<\omega$:

    \medskip{}

    ({*}{*}{*}) there exists a tuple $\overline{b}_{n}\in M_{b}^{k}$ and distinct elements $c_{1},\dots,c_{n}\in M_{t}$
    such that if $B_{n}=\beta^{M_{t}}[\overline{b}_{n}]$ then $B_{n}\subset U$
    and $(u,c_{i})\in X$ for all $u\in B_{n}$ and all $1\leq i\leq n$.

    \medskip{}

    According to the first claim there exists a $c_{1}\in M_{t}$ such
    that the definable set $(M_{t}\times\{c_{1}\})\cap X\cap(U\times M_{t})$
    is infinite. Therefore, its projection onto $U$ is infinite so there's
    some $\overline{b}_{1}\in M_{b}^{k}$ such that $B_{1}=\beta^{M_{t}}[\overline{b}_{1}]\subset U$
    is contained in the projection. This means that $(u,c_{1})\in X$
    for all $u\in B_{1}$ . This shows that ({*}{*}{*}) is true for $n=1$. 

    Now let's assume that ({*}{*}{*}) is true for $n\in\mathbb{N}$. The
    set \[
    \tilde{X}=\{(x,y)\in X:\forall1\leq i\leq n,y\neq c_{i}\}\]
    and the open definable set $B_{n}$ fulfill the conditions of the
    prior claim (where $\tilde{X}$ is instead of $X$ and $B_{n}$ is
    instead of $U$). This means that we can find an element $c_{n+1}\in M_{t}$
    such that \[
    \vert(M_{t}\times\{c_{n+1}\})\cap\tilde{X}\cap(B_{n}\times M_{t})\vert\ge\aleph_{0}\]
    
    So exactly like in the case of $n=1$, there exists some $\overline{b}_{n+1}\in M_{b}^{k}$
    such that $B_{n+1}=\beta^{M}[\overline{b}_{n+1}]\subset B_{n}$ and
    $(u,c_{n+1})\in\tilde{X}\subset X$ for all $u\in B_{n+1}$. Also,
    by the definition of $\tilde{X}$, $c_{n+1}\neq c_{i}$ for all $1\leq i\leq n$.
    Finally, since $B_{n+1}\subset B_{n}$, $(u,c_{i})\in X$ for all
    $u\in B_{n+1}$ and all $1\leq i\leq n+1$. So we showed that ({*}{*}{*})
    holds for all $n<\omega$.

    Therefore, if we define the formula:\[
    \gamma_{n}(\overline{x})=\exists c_{1}\dots\exists c_{n}((\bigwedge_{i\neq j}c_{i}\neq c_{j})\wedge(\forall u(\beta(u,\overline{x})\rightarrow((\bigwedge_{i}\phi(u,c_{i}))\wedge\alpha(u)))))\]
    then for each $n<\omega$ there exists a tuple $\overline{b}\in M_{b}^{k}$
    such that $M\vDash\gamma_{n}[\overline{b}]$. But $M$ is $\omega$-saturated
    so there is some $\overline{b}\in M_{b}^{k}$ such that $M\vDash\gamma_{n}[\overline{b}]$
    for all $n<\omega$, i.e, if $B=\beta^{M_{t}}[\overline{b}]$ then
    $B\subset U$ and the set $C=\{y\in M_{t}:\forall u\in B,(u,y)\in X\}$
    is infinite. This finishes the proof of the claim.
  \end{proof}

  Now, let $B$ and $C$ be the sets defined in the end of the proof
  of the second claim. Let $C_{0}$ be the (non empty) interior of $C$.
  Then by the definition of $C$, $B\times C_{0}\subset X\cap(U\times M_{t})$
  is open and which completes the proof of the proposition.
\end{proof}

\begin{lem}\label{lem:has-interior}
  Suppose $M$ is an $\omega$-saturated 1-t.t.t structure. Let $X\subset M_{t}^{n+1}$ be
  a definable subset such that $\pi_{1,\dots,n}(X)$ is a basis set
  in the product topology on $M_{t}^{n}$. If for all $\bar{x}\in \pi_{1,\dots,n}(X)$,
  $\vert(\{\bar{x}\}\times M_{t})\cap X\vert=\infty$, then $X$ has a
  non-empty interior.
\end{lem}
\begin{proof}
  We'll use induction on $n$.

  For $n=1$, this lemma follows proposition \ref{prop:2-d-nonempty-int}.

  Let's assume that the claim is true for $n-1$. We define $A=\pi_{2,\dots,n}(X)$
  and $B=\pi_{2,\dots,n+1}(X)$.

  In addition, we define the set:\[
  C=\{\bar{b}\in B:\:\vert(M_{t}\times\{\bar{b}\})\cap X\vert=\infty\}\]

  By proposition \ref{prop:2-d-nonempty-int}, for every tuple $\bar{a}\in A$
  the set $(M_{t}\times\{\bar{a}\}\times M_{t})\cap X$ has a non-empty
  interior. In particular, this means that 
  \[
  \vert(\{\bar{a}\}\times M_{t})\cap C\vert=\infty\]

  \begin{claim*}
    There exists a basis set $U\subset\pi_{1,\dots,n}(X)$ such that 
    \[
    \vert\{x\in M_{t}:\: U\times\{x\}\subset X\}\vert=\infty\]
  \end{claim*}
  \begin{proof}
    As we showed above, for every tuple $\bar{a}\in A$,
    \[
    \vert(\{\bar{a}\}\times M_{t})\cap C\vert=\infty\]
    
    By the inductive hypothesis, there exists a basis set $V\subset A$
    and a point $x_{1}\in M_{t}$ such that $V\times\{x_{1}\}\subset C$.
    By the definition of $C$ and another application of the inductive hypothesis,
    $(M_{t}\times V \times\{x_1\})\cap X$
    has a non-empty interior. Therefore, there exists a basis set $U_{1}\subset\pi_{1,\dots,n}(X)$
    such that $U_{1}\times\{x_{1}\}\subset X$.

    Now, lets define the set
    \[
    X_{2}=[(U_{1}\times M_{t})\cap X]\backslash[U_1\times\{x_{1}\}]\]
    
    Since we only removed a finite
    number of elements from each fiber of $U_1$, $X_{2}$ has the properties
    required by the proposition. This means that we can repeat the above
    process again and obtain a basis set $U_{2}\subset U_1=\pi_{1,\dots,n}(X_{2})$
    and an element $x_{2}\in M_{t}$ such that $x_1\neq x_2$ and 
    \[
    U_{2}\times\{x_{2}\}\subset X_2\subset X\]

    Furthermore, since $U_{2}\subset U_{1}$, we also have 
    \[U_{2}\times\{x_{1}\}\subset X\]
    Therefore: 
    \[
    \vert\{x\in M_{t}:\: U_{2}\times\{x\}\subset X\}\vert\geq2\]

    By continuing this process $n$ times, we can find a basis
    set $U_{n}\subset\pi_{1,\dots n}(X)$ such that:

    \[
    \vert\{x\in M_{t}:\: U_{n}\times\{x\}\subset X\}\vert\geq n\]

    Since $M$ is $\omega$-saturated and basis sets are definable with
    a tuple of constants from $M_{b}$, there exists a basis set
    $U\subset\pi_{1,\dots n}(X)$ such that:

    \[
    \vert\{x\in M_{t}:\: U\times\{x\}\subset X\}\vert=\infty\]

  \end{proof}
  Let $U$ be the basis set given by the claim. Since $M$
  is 1-t.t.t, there exists a basis set $W\subset M$ such that for all
  $w\in W$, $U\times\{w\}\subset X$.
  Therefore, $U\times W\subset X$ which finishes the induction and
  proves the lemma.
\end{proof}

Before proceeding to prove the the theorem about the equivalence of
the dimensions, we use lemma \ref{lem:has-interior} to obtain an
interesting corollary.

\begin{cor}\label{cor:complement-interior}
  Suppose $M$ is an $\omega$-saturated 1-t.t.t structure. Let $X\subset M_{t}^{n}$ and
  $Y\subset X$ be definable sets. If $X$ has an interior in $M_{t}^{n}$
  and $Y$ does not, then $X\backslash Y$ has an interior in $M_{t}^{n}$.
\end{cor}
\begin{proof}
  We use induction on $n$.

  For $n=1$, the lemma follows directly from the fact that $M$ is
  a 1-t.t.t structure.

  Let's assume the claim is true for $n$. Let $X\subset M_{t}^{n+1}$
  be a definable set with an interior and $Y\subset X$ be a definable
  set with no interior. In addition, we define $\tilde{X}=\pi_{1,\dots,n}(X)$
  , $\tilde{Y}=\pi_{1,\dots,n}(Y)\subset\tilde{X}$, and a set $\tilde{Z}\subset\tilde{Y}$:\[
  \tilde{Z}=\{\bar{y}\in\tilde{Y}:\:\vert(\{\bar{y}\}\times M_{t})\cap Y\vert=\infty\}\]
  
  Since $X$ has an interior, without loss of generality we can assume
  that for every $\bar{x}\in\tilde{X}$, $\vert(\{\bar{x}\}\times M_{t})\cap X\vert=\infty$.
  Furthermore, by lemma \ref{lem:has-interior}, $\tilde{Z}$ has no
  interior. So by the inductive hypothesis, $\tilde{U}=\tilde{X}\backslash\tilde{Z}$
  has an interior. 

  Let $\bar{u}$ be an element in $\tilde{U}$. Since $\vert(\{\bar{u}\}\times M_{t})\cap X\vert=\infty$
  and $\vert(\{\bar{u}\}\times M_{t})\cap Y\vert<\infty$, \[
  \vert(\{\bar{u}\}\times M_{t})\cap(X\backslash Y)\vert=\infty\]
  . So by lemma \ref{lem:has-interior}, $X\backslash Y$ has an interior
  in $M_{t}^{n+1}$.

  This completes the induction and the corollary.
\end{proof}

We can use the corollary to prove a proposition about dense definable
sets in $M_{t}^{n}$.

\begin{prop}\label{prop:dense-then-int-dense}
  Suppose $M$ is an $\omega$-saturated 1-t.t.t structure. Let $X\subset M_{t}^{n}$ be
  a dense definable set. Then $int(X)\subset M_{t}^{n}$ is dense as
  well. 
\end{prop}
\begin{proof}
  Let $a\in M_{t}^{n}$ be a point and $U\subset M_{t}$ a basis
  set containing $a$. Since $X$ is dense, $U\backslash X$ has an empty
  interior and so by corollary \ref{cor:complement-interior}, $U\cap X$
  has an interior. Therefore, there exists an element $b\in int(X)$
  such that $b\in U$. This finishes the proof.
\end{proof}

Each of following two propositions will be used to show one of the 
inequalities which together will prove the equivalence of the dimensions.

\begin{prop}\label{pro:rank<dim}
  Suppose $M$ is an $\omega$-saturated 1-t.t.t structure. Let $X\subset M_{t}^{n}$ be
  definable over $A$, $0\leq k\leq n$, $1\leq i_{1}<\dots<i_{k}\leq n$
  and $\bar{a}\in X$ such that $(a_{i_{1}},\dots,a_{i_{k}})$ is algebraically
  independent over $A$. Then $\pi_{i_{1},\dots,i_{k}}(X)$ has an interior.
\end{prop}
\begin{proof}
  We use induction on $n$.

  If $n=1$, then since $a_{i_{1}}\notin acl(A)$, $X$ is infinite
  and thus has an interior.

  Let's assume the claim holds for $n$.

  First we assume that $i_{k}<n+1$. In this case, the claim follows
  directly from the inductive hypothesis.

  Now let's assume that $i_{k}=n+1$. We define $Y=\pi_{i_{1},\dots,i_{k-1},n+1}(X)$
  , $Z=\pi_{i_{1},\dots,i_{k-1}}(X)=\pi_{i_{1},\dots,i_{k-1}}(Y)$, and:\[
  C=\{\bar{z}\in Z:\:\vert(\{\bar{z}\}\times M_{t})\cap Y\vert=\infty\}\]
  . Since $(a_{i_{1}},\dots,a_{i_{k}})$ is algebraically independent
  over $A$, $(a_{i_{1}},\dots,a_{i_{k-1}})\in C$. So by the inductive
  hypothesis, $C$ has a non-empty interior and by lemma \ref{lem:has-interior},
  $Y$ has an interior.

  This completes the induction and the proposition.
\end{proof}

\begin{prop}\label{pro:dim<rank}
  Suppose $M$ is an $\omega$-saturated 1-t.t.t structure. Let $X\subset M_{t}^{n}$ be
  definable over $A$, $0\leq k\leq n$, and $1\leq i_{1}<\dots<i_{k}\leq n$
  such that $\pi_{i_{1},\dots,i_{k}}(X)$ has an interior. Then there
  exists an elementary extension $M\prec N$ and a tuple $\bar{a}\in X(N)$ such that $(a_{i_{1}},\dots,a_{i_{k}})$
  is algebraically independent over $A$.
\end{prop}
\begin{proof}
  We use induction on $n$.

  Let $n=1$ and $X\subset M_{t}$ a subset definable over $A$. For
  $k=0$ there's nothing to show. 

  If $k=1$ then $X$ is infinite so by compactness there exists an elementary extension $M\prec N$
  and an element $x\in X(N)$ such that $x\notin acl(A)$. Therefore, we can take
  $a_{i_{1}}=x$.

  Let's assume the claim is true for $n$. Let $X\subset M_{t}^{n+1}$
  be definable over $A$ and $0\leq k\leq n+1$ such that $dim(X)=k$
  and $\pi_{i_{1},\dots,i_{k}}(X)$ has an interior.

  First we assume that $i_{k}<n+1$.

  Let's define $Y=\pi_{1,\dots n}(X)$. According to the assumption,
  $\pi_{i_{1},\dots,i_{k}}(Y)$ has an interior. So by the inductive
  hypothesis, there exists an elementary extension $M\prec N$ and a tuple $\bar{y}\in Y(N)$ 
  such that $(y_{i_{1}},\dots,y_{i_{k}})$
  is algebraically independent over $A$. Since $\bar{y}\in Y(N)$, there
  exists an element $x\in N_{t}$ such that $\bar{a}\mathcircumflex x\in X(N)$.

  Now let's assume that $i_{k}=n+1$.

  Let's define $Y=\pi_{i_{1},\dots,i_{k-1},n+1}(X)$. According to the
  assumption, $Y$ has a non-empty interior. This means that there exist
  basis sets $U\subset\pi_{i_{1},\dots,i_{k-1}}(X)$ and $V\subset M_{t}$
  such that $U\times V\subset Y$. According to the inductive hypothesis,
  there exists an elementary extension $M\prec N$ and a tuple $\bar{u}=(u_{1},\dots,u_{k})\in U(N)$ such that
  $(u_{1},\dots,u_{k})$ is algebraically independent over $A$. In
  addition, since $V$ is infinite, we can find an elementary extension $N\prec N'$ and an element $v\in V(N')$
  such that $v\notin acl(\bar{u}/A)$. By lemma \ref{lem:independent},
  $\bar{u}\mathcircumflex v\in A$ is algebraically independent over
  $A$.

  This finishes the induction and proves the proposition.
\end{proof}

\begin{thm}\label{thm:rank=dim}
  Suppose $M$ is an $\omega$-saturated 1-t.t.t structure. Let $X\subset M_{t}^{n}$ be
  definable. Then $rk(X)=dim(X)$.
\end{thm}

\begin{proof}
  Let's assume that $X$ is definable over $A$. 

  We first prove that $rk(X)\leq dim(X)$.

  Let's set $0\leq k\leq n$ such that $rk(X)=k$. By the definition of $rk(X)$ (see remark \ref{rem:rank-rem}),
  there exists an elementary extension $M\prec N$, a tuple $\bar{a}\in X(N)$, and indices $1\leq i_{1}<\dots<i_{k}\leq n$
  such that $(a_{i_{1}},\dots,a_{i_{k}})$ is algebraically independent
  over $A$. 

  By \cite[14]{D}, $N$ is also a 1-t.t.t structure.

  Therefore, by proposition \ref{pro:rank<dim}, $\pi_{i_{1},\dots,i_{k}}(X(N))$ has an interior
  which means that $\pi_{i_{1},\dots,i_{k}}(X)$ has an interior as well.
  So by the definition of $dim(X)$, $dim(X)\geq k$.

  We now prove that $rk(X)\geq dim(X)$.

  Let's set $0\leq k\leq n$ such that $dim(X)=k$. By the definition of $dim(X)$,
  there exist indices $1\leq i_{1}<\dots<i_{k}\leq n$ such that $\pi_{i_{1},\dots,i_{k}}(X)$
  has an interior. Therefore, by proposition \ref{pro:dim<rank}, there
  exists an elementary extension $M\prec N$ and a tuple $\bar{a}\in X(N)$ such that $(a_{i_{1}},\dots,a_{i_{k}})$
  is algebraically independent over $A$. So by the definition of $rk(X)$,
  $rk(X)\geq k$.

  So together, we proved that $rk(X)=dim(X)$.
\end{proof}

\subsection*{Acknowledgement}
  I'd like to thank Ehud Hrushovski for many helpful discussions.

\bibliographystyle{jflnat}
\bibliography{ttt_dimension}

\begin{thebibliography}{6}
\expandafter\ifx\csname natexlab\endcsname\relax\def\natexlab#1{#1}\fi
\def\docolon{:}
\def\eatcomma#1{}
\def\onlyone#1{\gdef\oneletter{#1}}
\def\sphref#1#2{{\let\#=\docolon\xdef\one{#1}}\href{\one}{#2}}
\def\zhref#1,#2{{\let\#=\docolon\xdef\one{#1}}\href{\one}{#2}}
\expandafter\ifx\csname url\endcsname\relax
  \def\url#1{{\tt #1}}\fi
\newcommand{\enquote}[2]{``#1,''}

\bibitem[A.Pillay(1986)]{P3}
A.Pillay,
\newblock \enquote{Some remarks on definable equivalence relations in o-minimal
  structures},
\newblock {\em The Journal of Symbolic Logic}, vol.~51 (1986),
  pp.~709--714.\eatcomma.

\bibitem[A.Pillay(1988)]{P2}
A.Pillay,
\newblock \enquote{On groups and fields definable in o-minimal structures},
\newblock {\em J. Pure Appl. Algebra}, vol.~53 (1988), pp.~239--255.\eatcomma.

\bibitem[D.Lowengrub(2012)]{D}
D.Lowengrub,
\newblock \enquote{One dimensional t.t.t structures},
\newblock {\em pre-print},  (2012).\eatcomma. arXiv:1210.5603.

\bibitem[Knight et~al.(1986)Knight, Pillay, and Steinhorn]{KPS}
Knight, J., A.~Pillay \unskip, and  C.~Steinhorn,
\newblock \enquote{Definable sets in ordered structures ii},
\newblock {\em Trans. Amer. Math Soc.}, vol.~295 (1986),
  pp.~593--605.\eatcomma.

\bibitem[Mathews(1995)]{M}
Mathews, L.,
\newblock \enquote{Cell decomposition and dimension functions in first-order
  topological structures},
\newblock {\em Proc. London Math Soc.}, vol.~70 (1995), pp.~1--32.\eatcomma.

\bibitem[Pillay(1987)]{P1}
Pillay, A.,
\newblock \enquote{First order topological structures and theories},
\newblock {\em The Journal of Symbolic Logic}, vol.~52 (1987),
  pp.~763--778.\eatcomma.

\end{thebibliography}

\end{document}